\documentclass[reqno]{amsart}

\usepackage{amsthm,amsfonts,amsmath,mathrsfs,hyperref,amssymb,xypic}

\title[The Theta Correspondence, Periods and Special Values of $L$-Functions]{The Theta Correspondence, Periods of Automorphic Forms and Special Values of Standard Automorphic $L$-Functions}
\author{Patrick Walls}
\date{\today}

\newcommand{\Z}{\mathbb{Z}}
\newcommand{\Zp}{\mathbb{Z}_p}
\newcommand{\Q}{\mathbb{Q}}
\newcommand{\Qp}{\mathbb{Q}_p}
\newcommand{\A}{\mathbb{A}}

\newcommand{\C}{\mathbb{C}}

\newcommand{\OV}{\mathrm{O}(V)}
\newcommand{\Sp}{\mathrm{Sp}}

\newcommand{\Symn}{\mathrm{Sym}_n}
\newcommand{\Spn}{\mathrm{Sp}_n}

\newcommand{\bv}{\boldsymbol{v}}

\newcommand{\bx}{\boldsymbol{x}}

\newcommand{\leftexp}[2]{{\vphantom{#2}}^{#1} \hspace{-.6mm} {#2}}

\newtheorem{theorem}{Theorem}[subsection]

\newtheorem{proposition}{Proposition}[subsection]
\newtheorem{lemma}{Lemma}[subsection]
\newtheorem*{lemma*}{Lemma}

\numberwithin{equation}{subsection}
\begin{document}

\maketitle

\tableofcontents

\section*{Introduction}

The zeros and poles of standard automorphic $L$-functions attached to representations of classical groups are linked to the nonvanishing of lifts in the theory of the theta correspondence. This deep and subtle connection is formulated in terms of the Rallis inner product formula, the Siegel-Weil formula and its extensions, and was developed over several years by many authors (for symplectic-orthogonal dual pairs see \cite{R2}, \cite{R1}, \cite{R3}, \cite{KR-SW1}, \cite{KR-SW2}, \cite{KR}, for other dual pairs see \cite{Ik}, \cite{I1}, \cite{I2}, \cite{I3}, \cite{Y1} and \cite{Y2}, and Weil's original work \cite{Weil2}). We refer to the recent work \cite{GQT} for an excellent account of this theory as well as the proof of the last remaining piece of the regularized Siegel-Weil formula in the second term range.

In this paper, we take the following as our point of departure. Let $\sigma$ be a cuspidal automorphic representation of the symplectic group $\Spn$ defined over a number field $F$. There is an entire ocean of orthogonal groups to which $\sigma$ may lift by the theta correspondence. We may imagine the collection of orthogonal groups as a collection of Witt towers (cf.~\autoref{Tower}) with orthogonal groups attached to anisotropic quadratic spaces of all different dimensions and characters at their bases. Rallis's tower property shows that $\sigma$ must lift to a nonzero cuspidal representation at some index in each Witt tower and Rallis introduced his inner product formula to derive nonvanishing criteria in terms of automorphic $L$-functions.

The results of this paper show that when a cuspidal representation $\sigma$ of $\Spn$ \emph{does} lift to a cuspidal representation $\pi = \theta_{\psi}(\sigma)$ of an orthogonal group $\OV$ (where\footnote{We have chosen to restrict to symplectic-orthogonal duals pairs $(\Spn,\OV)$ where $\mathrm{dim}_F \, V = m$ is \emph{even} for the sake of convenience and clarity. Our results apply with minor modifications to the case when $m$ is odd and $G$ is the metaplectic cover of $\Spn$, as well as to the case for unitary dual pairs.} $\mathrm{dim}_F \, V = m$ is \emph{even}), the Fourier coefficients of automorphic forms in $\sigma$ are linked to the orthogonal periods of automorphic forms in $\pi$ (cf.~\autoref{MainTheorem} and \autoref{InnerProdThm}). Consequently, when our results are combined with the Rallis inner product formula in the convergent range or the second term range (ie.~when $m > 2n+1$ or $V$ is anisotropic, cf.~(\ref{Convergent}) and (\ref{SecondTerm})), we prove a special value formula (cf.~\autoref{SpecialValue}) for the standard automorphic $L$-function $L(s,\sigma,\chi_V)$ attached to $\sigma$ (and twisted by the character $\chi_V$ of $V$) at the point $s_{m,2n}=m/2-(2n+1)/2$ in terms of the Fourier coefficients of automorphic forms in $\sigma$ and orthogonal periods of forms in $\pi$. \\

\noindent {\bf Notation.} We use the following notation throughout the paper. Let $F$ be a number field with ring of integers $O_F$ and ring of adeles  $\A = \A_F$. We fix an additive unitary class character $\psi : \A / F \longrightarrow \C^{\times}$. Let $\A_{\infty} = \prod_{v | \infty} F_v$ be the product of the completions of $F$ at the archimedean places and let $\A_f = \prod'_{v < \infty} F_v$ be the ring of finite adeles (ie.~the restricted product of the completions of $F$ at nonarchimedean places with respect to the rings of integers $O_{F_v}$). For an algebraic group $G$ over $F$, let $[G] = G(F) \backslash G(\A)$ be the adelic quotient, $dg$ is any fixed Haar measure on $G(\A)$ and we write $\mathrm{vol} \, [G]$ for the measure of $[G]$ with respect to this measure. The results of this paper are independent of the choice of measures. The inner product of square-integrable functions on $[G]$ is written $\langle f_1,f_2 \rangle = \int_{[G]} f_1(g) \overline{f_2(g)} \, dg$. Finally, $\leftexp{t}{A}$ is the transpose of a matrix $A$ and $| \cdot |_{\A} = \prod_v | \cdot |_v$ is the adelic norm.

\section{The Theta Correspondence}

Let $F$ be a number field with $\A = \A_F$ its ring of adeles and fix once and for all an additive unitary class character $\psi : \A / F \longrightarrow \C^{\times}$. In this section, we introduce the theta correspondence and the Weil representation for symplectic-orthogonal dual pairs, the Rallis tower property and the cuspidality of theta lifts. This material is standard and so our treatment will be brief. We refer to \cite{MVW}, \cite{KudlaTheta} and \cite{Rao} for an introduction to the Weil representation and the theta correspondence (see also Weil's original paper \cite{Weil1}) and Rallis's original work \cite{R1} for his tower property.

\subsection{Symplectic and Orthogonal Groups} \label{SecGroups}

Let $V$ be a vector space over $F$ such that $\mathrm{dim}_F \hspace{0.1mm} V = m$ is $even$ and let $( \ , \hspace{0.8mm} )$ be a nondegenerate symmetric bilinear form on $V$. Let $H=\OV$ be the orthogonal group over $F$ attached to the pair $V , ( \ , \hspace{0.7mm} )$.

Let $W$ be the standard vector space over $F$ of dimension $2n$ equipped with the nondegenerate skew-symmetric bilinear form $\langle \ , \hspace{0.5mm} \rangle$ defined by the matrix
$$
J = \begin{pmatrix} 0 & 1_n \\ -1_n & 0 \end{pmatrix}
$$
where $1_n$ is the identity matrix of size $n$. Let $G=\Spn$ be the symplectic group over $F$ attached to the pair $W,\langle \ , \hspace{0.5mm} \rangle$. In particular, we view $W$ as the space of row vectors of size $2n$ and the automorphisms of $W$ as matrices acting by right multiplication therefore
\begin{equation}
G=\Spn = \left\{ g \in \mathrm{GL}_{2n} : g \, J \, \leftexp{t}{g} = J \right\} \, .
\end{equation}

The Siegel parabolic subgroup $P \subset \Spn$ has a decomposition $P=MN$ where the Levi component $M$ is
\begin{equation}
M = \left\{ m(a) = \begin{pmatrix} a & 0 \\ 0 & \leftexp{t}{a}^{-1} \end{pmatrix} : a \in \mathrm{GL}_n \right\}
\end{equation}
and the unipotent radical $N$ is
\begin{equation}
N =\left\{ n(b) = \begin{pmatrix} 1 & b \\ 0 & 1 \end{pmatrix} : b \in \mathrm{Sym}_n \right\}
\end{equation}
where $\mathrm{Sym}_n$ denotes symmetric matrices of size $n$.

We restrict to the case that $m$ is even for convenience and clarity. The experienced reader will see that the results contained in this paper extend with minor modifications to the case where $m$ is odd and $G$ is the metaplectic cover of $\Spn$. 

\subsection{The Weil Representation} \label{SecWeil}

Attached to the dual pair $(G,H)$ is a Weil representation $\omega = \omega_{\psi}$ (relative to the fixed character $\psi$) which is an action of $G(\A) \times H(\A)$ on the space $S(V(\A)^n)$ of Schwartz functions on
\begin{equation}
V(\A)^n = \{ \bv = (v_1, \dots , v_n) : v_1 , \dots , v_n \in V \otimes_F \A \} \ .
\end{equation}
This model of the Weil representation satisfies the familiar formulas for $\varphi \in S(V(\A)^n)$
\begin{equation}
\begin{array}{rlc}
\omega(h)\varphi(\bv) & \hspace{-2.5mm} = \varphi(h^{-1} \bv) & \hspace{5mm} h \in H(\A) \\
\omega(m(a))\varphi(\bv) & \hspace{-2.5mm} = \chi_V(\mathrm{det}\, a) \, |\mathrm{det}\, a |_{\A}^{m/2} \, \varphi(\bv a) & \hspace{5mm} m(a) \in M(\A) \\
\omega(n(b))\varphi(\bv) & \hspace{-2.5mm} = \psi(\mathrm{tr} \, b Q[\bv]) \, \varphi(\bv) & \hspace{5mm} n(b) \in N(\A)
\end{array}
\end{equation}
where
\begin{equation}
\chi_V(x) = \left( x,(-1)^{\frac{m(m-1)}{2}} \mathrm{det} \, V \right)_{\A}
\end{equation}
is the character of the quadratic space $V$ and $( \ , \hspace{0.7mm} )_{\A}$ is the global Hilbert symbol, and
\begin{equation}
Q[\bv]=\frac{1}{2} ((v_i,v_j ))_{i,j} \in \Symn(\A)
\end{equation}
where $((v_i,v_j ))_{i,j}$ is the symmetric matrix of inner products of the components of $\bv = (v_1,\dots,v_n)$.

\subsection{The Theta Correspondence}

Each $\varphi \in S(V(\A)^n)$ defines a theta function
\begin{equation}
\theta_{\varphi}(g,h) = \sum_{\bv \in V^n} \omega(g)\varphi(h^{-1}\bv) \ , \ g \in G(\A) \ , \ h \in H(\A)
\end{equation}
which is an automorphic form on the product $G(\A) \times H(\A)$. In particular, $\theta_{\varphi}$ is left invariant by $G(F) \times H(F)$ and is of moderate growth. Note that the formation of theta functions defines a ($G(\A) \times H(\A)$)-equivariant map from the Weil representation $\omega$ to the space $\mathscr{A}(G \times H)$ of automorphic forms on $G(\A) \times H(\A)$
\begin{equation}
\omega \longrightarrow \mathscr{A}(G \times H) : \varphi \mapsto \theta_{\varphi} \ .
\end{equation}

Theta functions serve as kernel functions for the theta correspondence which map automorphic forms on one group to automorphic forms on the other. In particular, let $f^H$ be a cuspidal automorphic form on $H$, let $\varphi \in S(V(\A)^n)$ and define the theta lift of $f^H$ to $G$ by
\begin{equation}
\theta_{\varphi}f^H(g) = \int_{[H]} \theta_{\varphi}(g,h) \, f^H(h) \, dh \ .
\end{equation}
This integral is convergent since $f^H$ is cuspidal and $\theta_{\varphi}$ is of moderate growth. For a cuspidal automorphic representation $\pi$ of $H$, the theta lift of $\pi$ to $G$ is the space $\theta_{\psi}(\pi)$ of all theta lifts $\theta_{\varphi}f^H$ for $\varphi \in S(V(\A)^n)$ and $f^H \in \pi$.

Similarly, let $f^G$ be a cuspidal automorphic form on $G$, let $\varphi \in S(V(\A)^n)$ and define the theta lift of $f^G$ to $H$ by
\begin{equation}
\theta_{\varphi^{\vee}}f^G(h) = \int_{[G]} \overline{\theta_{\varphi}(g,h)} \, f^G(g) \, dg \ .
\end{equation}
Again, this integral is convergent since $f^G$ is cuspidal and $\theta_{\varphi}$ is of moderate growth. For a cuspidal automorphic representation $\sigma$ of $G$, the theta lift of $\sigma$ to $H$ is the space $\theta_{\psi}(\sigma)$ of all theta lifts $\theta_{\varphi^{\vee}}f^G$ for $\varphi \in S(V(\A)^n)$ and $f^G \in \sigma$.

Our definitions of theta lifts imply the following adjoint property: if $f^G$ and $f^H$ are cuspidal automorphic forms on $G$ and $H$ respectively and $\varphi \in S(V(\A)^n)$ such that \emph{both} theta lifts $\theta_{\varphi^{\vee}}f^G$ and $\theta_{\varphi}f^H$ are cuspidal (ensuring that every integral below is absolutely convergent), then
\begin{equation} \label{adjoint}
\left\langle f^G , \theta_{\varphi}f^H \right\rangle_G = \left\langle \theta_{\varphi^{\vee}}f^G , f^H \right\rangle_H \ .
\end{equation}
Furthermore, we know by the work of Moeglin \cite{Moeglin} (generalizing \cite{R1}) that if $\sigma$ is a cuspidal representation of $G$ such that its theta lift consists of cusp forms, then $\pi = \theta_{\psi}(\sigma)$ is an irreducible cuspidal automorphic representation of $H$ and
\begin{equation}
\theta_{\psi}(\theta_{\psi}(\sigma)) = \sigma \ .
\end{equation}
Finally, the analogous statement for representations of $H$ also holds.

\subsection{Cuspidality and Rallis's Tower Property} \label{Tower}

Natural questions arise immediately: When is a theta lift nonzero? When is a theta lift cuspidal? These questions were first addressed by Rallis and the tower property.

Let $(V_0,Q_0)$ be an anisotropic quadratic space over $F$ and let $H_0$ be the corresponding orthogonal group. Let $\mathbb{H}$ be the hyperbolic space of dimension $2$ equipped with the quadratic form $Q(x,y)=xy$. For each $r \geq1$, define the quadratic space
\begin{equation}
V_r = V_0 \oplus \underbrace{\mathbb{H} \oplus \cdots \oplus \mathbb{H}}_{r \text{ copies}}
\end{equation}
and let $H_r$ be the corresponding orthogonal group. The increasing sequence of groups
\begin{equation} \label{WittTower}
H_0 \subset H_1 \subset H_2 \subset \cdots \subset H_r \subset \cdots
\end{equation}
is called the Witt tower attached to $H_0$. There is a Weil representation for each dual pair $(G,H_r)$ and we let $\theta_{\psi, r}(\sigma)$ denote the theta lift of a cuspidal representation $\sigma$ of $G(\A)$ to $H_r(\A)$. Rallis's tower property is the following.

\begin{theorem}[\cite{R1}]
Let $\sigma$ be a cuspidal automorphic representation of $G(\A)$ and let $i$ be the smallest integer such that $\theta_{\psi,i}(\sigma)$ is nonzero. Then:
\begin{enumerate}
\item $i \leq 2n$,
\item $\theta_{\psi, i}(\sigma)$ is cuspidal,
\item $\theta_{\psi,r}(\sigma)$ is nonzero for all $r \geq i$.
\end{enumerate}
Furthermore, the space of cusp forms on $G(\A)$ decomposes into the orthogonal sum
$$
L^2_{\mathrm{cusp}}(G(F) \backslash G(\A)) = I(Q_0) \oplus I(Q_1) \oplus \cdots \oplus I(Q_{2n})
$$
where $I(Q_r)$ is the space of $\sigma$'s such that $\theta_{\psi,r}(\sigma)\not=0$ and $\theta_{\psi,r'}(\sigma)=0$ for $r' < r$.
\end{theorem}

\noindent The integer $i$ in the theorem is called the first occurrence index of $\sigma$ in the Witt tower of $H_0$. Note that Rallis's tower property shows that a cuspidal representation $\sigma$ of $G$ must lift to a cuspidal representation $\theta_{\psi,i}(\sigma)$ of $H_i$ for some $0 \leq i \leq 2n$ but the theorem does not give any information about the index itself.

The goal of this paper is to show that there are general relations between periods of automorphic forms of $\sigma$ and $\pi$ in the case that $\sigma$ is a cuspidal representation of $G$ such that $\pi = \theta_{\psi}(\sigma)$ is a cuspidal representation of $H$.

\section{Periods of Automorphic Forms}

The main result of this paper is the period identity in \autoref{MainTheorem} which shows a general relation between Fourier coefficients and orthogonal periods of cuspidal automorphic forms which correspond by the theta correspondence. As a consequence, we prove in the last section of this paper a special value formula for standard automorphic $L$-functions by combining our identity with the Rallis inner product formula (cf.~\autoref{SpecialValue}).

We begin by introducing Fourier coefficients of automorphic forms on symplectic groups and periods of cuspidal automorphic forms on orthogonal groups. The connection between Fourier coefficients and orthogonal periods in the theta correspondence is shown in \autoref{DualPeriods} from which we derive \autoref{MainTheorem}.

\subsection{Fourier Coefficients and Orthogonal Periods}

Let $G=\Sp_n$ and $H=\OV$ as in \autoref{SecGroups}. The symplectic group $G$ contains the Siegel parabolic subgroup $P=MN$ with unipotent radical
\begin{equation} \label{Unipotent}
N = \left\{ n(b)=\begin{pmatrix} 1 & b \\ 0 & 1 \end{pmatrix} : b \in \Symn \right\}
\end{equation}
where $\Symn$ denotes the space of symmetric matrices of size $n$. For $T \in \Symn(F)$, the $\psi_{T}$-Fourier coefficient of an automorphic form $f^G$ on $G$ is
\begin{equation}
W_{T}(f^G) = \int_{[N]} f^G(n) \, \overline{\psi_{T}(n)} \, dn
\end{equation} 
where $\psi_{T}(n)=\psi(\mathrm{tr} \, T b)$ for $n=n(b)$ as in (\ref{Unipotent}) above.

Let $\bx = (x_1,\dots,x_n) \in V^n$ be an $n$-tuple of rational vectors in $V$ and let $H_{\bx}$ be the stabilizer of $\bx$ in $H$ where $H$ acts on $V^n$ componentwise. The $H_{\bx}$-period of a cuspidal automorphic form $f^H$ on $H$ is
\begin{equation}
P_{H_{\bx}}(f^H) = \int_{[H_{\bx}]} f^H(h) \, dh \ .
\end{equation}

Fourier coefficients and orthogonal periods are related by the following (well-known) calculation. Given a cuspidal automorphic form $f^H$ on $H$ and a Schwartz function $\varphi \in S(V(\A)^n)$, the $\psi_{T}$-Fourier coefficient of the theta lift $\theta_{\varphi}f^H$ is
\begin{align*}
W_{T}(\theta_{\varphi} f^H) &= \int_{[N]} \left( \int_{[H]} \sum_{\bv \in V^n} \omega(n)\varphi(h^{-1}\bv) \, f^H(h) \, dh \right) \overline{\psi_{T}(n)} \, dn \\
&= \int_{[H]} \left( \int_{[\Symn]} \sum_{\bv \in V^n} \varphi(h^{-1}\bv) \, \psi( \mathrm{tr}(bQ[\bv]-T b)) \, db \right) f^H(h) \, dh \\
&= \int_{[H]} \sum_{\scriptsize{ \begin{array}{c} \bv \in V^n \\ Q[\bv]=T \end{array} } } \varphi(h^{-1}\bv) \, f^H(h) \, dh
\end{align*}
If we assume that $T$ is nondegenerate, then the group $H(F)$ acts transitively on the set in the sum above and so we chose some $\bx \in V^n$ such that $Q[\bx]=T$ and we let $H_{\bx}$ be its stabilizer. If there is no such $\bx$ (in other words, if $Q$ does not represent $T$), then $W_{T}(\theta_{\varphi} f^H) = 0$ for all $f^H$. We continue
\begin{align*}
W_{T}(\theta_{\varphi}f^H) &= \int_{[H]} \sum_{\gamma \in H_{\bx}(F) \backslash H(F)} \varphi(h^{-1} \gamma^{-1} \bx) \, f^H(h) \, dh \\
&= \int_{H_{\bx}(F) \backslash H(\A)} \varphi(h^{-1} \bx) \, f^H(h) \, dh \\
&= \int_{H_{\bx}(\A) \backslash H(\A)} \varphi(h^{-1} \bx) \int_{H_{\bx}(F) \backslash H_{\bx}(\A)} f^H(h' h) \, dh' \, dh \ .
\end{align*}
Thus the $\psi_{T}$-Fourier coefficient of the theta lift of $f^H$ is written in terms of a $H_{\bx}$-period where $Q[\bx]=T$. To push this integral one step further, we introduce the following lemma.

\begin{lemma} \label{MatchingLemma}
Let $\bx \in V^n$ such that $Q[\bx]=T$ is nondegenerate and let $H_{\bx}$ be the stabilizer of $\bx$ in $H$. For each $\varphi \in S(V(\A)^n)$, there is a smooth function $\xi$ on $H(\A)$ which is rapidly decreasing on $H(\A_{\infty})$ and compactly supported on $H(\A_f)$ satisfying
\begin{equation}
\varphi(h^{-1}\bx) = \int_{H_{\bx}(\A)} \xi(h_0 h) \, dh_0 \ .
\end{equation}
In this case, we say $(\varphi, \xi ; \bx)$ is a matching datum (or that $\xi$ matches $\varphi$ relative to $\bx$).
\end{lemma}

We will defer the proof of the lemma until the next section. Continuing with the integral above with the function $\xi$ matching $\varphi$ relative to $\bx$ as in the previous lemma, we have
\begin{align*}
W_{T}(\theta_{\varphi}f^H) &= \int_{H_{\bx}(\A) \backslash H(\A)} \varphi(h^{-1} \bx) \int_{H_{\bx}(F) \backslash H_{\bx}(\A)} f^H(h' h) \, dh' \, dh \\
&= \int_{H_{\bx}(\A) \backslash H(\A)} \int_{H_{\bx}(\A)} \xi(h_0 h) \, dh_0 \int_{H_{\bx}(F) \backslash H_{\bx}(\A)} f^H(h' h) \, dh' \, dh \\
&= \int_{H(\A)} \xi(h) \int_{H_{\bx}(F) \backslash H_{\bx}(\A)} f^H(h' h) \, dh' \, dh \\
&= \int_{H_{\bx}(F) \backslash H_{\bx}(\A)} \int_{H(\A)} \xi(h) \, f^H(h' h) \, dh \, dh' \ .
\end{align*}

Therefore we have proved the following formal identity which is the main input into our period identities \autoref{MainTheorem} and \autoref{InnerProdThm}. In the next section, we prove \autoref{MatchingLemma} and give an explicit formula for the local functions $\xi_v$ at finite places.

\begin{proposition} \label{DualPeriods}
Let $\bx \in V^n$ such that $Q[\bx]=T$ is nondegenerate and let $H_{\bx}$ be the stabilizer of $\bx$ in $H$. Let $\varphi \in S(V(\A)^n)$ and let $\xi$ be a smooth function which matches $\varphi$ relative to $\bx$ as in \autoref{MatchingLemma}. For any cusp form $f^H \in L^2_{\mathrm{cusp}}(H(F) \backslash H(\A))$, we have
\begin{equation}
W_{T}(\theta_{\varphi}f^H) = P_{H_{\bx}}(R_{\xi}f^H)
\end{equation}
where $R_{\xi}f^H(h) = \int_{H(\A)} \xi(h') \, f^H(hh') \, dh'$.
\end{proposition}

\subsection{Matching Functions}

In this section, we will prove \autoref{MatchingLemma} along with an explicit formula  for the local functions $\xi_v$ at finite places.

\begin{lemma*}
Let $\bx \in V^n$ such that $Q[\bx]=T$ is nondegenerate and let $H_{\bx}$ be the stabilizer of $\bx$ in $H$. For each $\varphi \in S(V(\A)^n)$, there is a smooth function $\xi$ on $H(\A)$ which is rapidly decreasing on $H(\A_{\infty})$ and compactly supported on $H(\A_f)$ satisfying
\begin{equation} \label{MatchingEq1}
\varphi(h^{-1}\bx) = \int_{H_{\bx}(\A)} \xi(h_0 h) \, dh_0 \ .
\end{equation}
In this case, we say $(\varphi, \xi ; \bx)$ is a matching datum (or that $\xi$ matches $\varphi$ relative to $\bx$). Furthermore, if $\varphi = \otimes_v \varphi_v$ is factorizable, then $\xi = \otimes_v \xi_v$ is factorizable where, for any finite place $v$, the local function $\xi_v$ is given by the finite sum
\begin{equation} \label{MatchingEq2}
\xi_v(h) = \sum_{ {\scriptsize \begin{array}{c} \gamma \in \mathscr{C} \\ \gamma^{-1} \bx \in \mathrm{supp} \, \varphi_v \end{array} } } \frac{\varphi_v(\gamma^{-1} \bx)}{\mathrm{vol}(H_{\bx}(F_v) \cap \gamma U_v \gamma^{-1})} \ C_{\gamma U_v}(h)
\end{equation}
where
\begin{enumerate}
\item $U_v \subset H(F_v)$ is any compact open subgroup which acts trivially on $\varphi_v$
\item $\mathscr{C}$ is a set of representatives of the double coset space $H_{\bx}(F_v) \backslash H(F_v) / U_v$
\item $C_{\gamma U_v}$ is the characteristic function of $\gamma U_v$.
\end{enumerate}
\end{lemma*}

Let us make few remarks before we continue with the proof of this lemma. Notice that if $\xi(h)$ is any function which satisfies (\ref{MatchingEq1}), then $\xi(h'h)$ also satisfies (\ref{MatchingEq1}) for any $h' \in H_{\bx}(\A)$. Also, the local function $\xi_v$ described in (\ref{MatchingEq2}) \emph{depends on the choice $\mathscr{C}$ of representatives} for the double coset space in the sum. Finally, although the set $\mathscr{C}$ may be infinite, the sum is always finite.

\begin{proof}
We may assume that $\varphi = \otimes_v \, \varphi_v$ is factorizable therefore we will prove the analogous equality
\begin{equation}
\varphi_v(h^{-1}\bx) = \int_{H_{\bx}(F_v)} \xi_v(h_0h) \, dh_0 \ , \ h \in H(F_v)
\end{equation}
for each place $v$. The group $H(\A)$ can be written as the restricted product $\prod_v' H(F_v)$ with respect to the compact open subgroups $K_v = \mathrm{Aut}(L_v) \subset H(F_v)$ where $L_v = L \otimes_{O_F} O_{F_v}$ for a fixed global $O_F$-lattice $L \subset V$. We will show that $\xi_v$ is the characteristic function of $K_v$ for almost every place $v$.

Suppose $v$ is a finite place. If $V_v$ is anisotropic, the group $H(F_v)$ is compact therefore
\begin{equation}
\xi_v(h)= \frac{1}{\mathrm{vol} \, H_{\bx}(F_v) } \, \varphi_v(h^{-1} \bx) \ , \ h \in H(F_v)
\end{equation}
is a smooth compactly supported function which matches $\varphi_v$ relative to $\bx$.

Suppose $V_v$ is isotropic. Let $U_v \subset H(F_v)$ be a compact open subgroup such that $\varphi_v(k \bv) = \varphi_v(\bv)$ for all $k \in U_v$ and $\bv \in V^n_v$. Use the notation $C_S$ to denote the characteristic function of a set $S$ and write
\begin{align*}
\varphi_v(h^{-1} \bx) &= \sum_{ {\scriptsize \begin{array}{c} \gamma \in H(F_v) / U_v \\ \gamma^{-1} \bx \in \text{supp} \, \varphi_v \end{array} } } \varphi_v(\gamma^{-1} \bx) \, C_{\gamma U_v }(h) \\
&= \sum_{ {\scriptsize \begin{array}{c} \gamma \in H_{\bx}(F_v) \backslash H(F_v) / U_v \\ \gamma^{-1} \bx \in \text{supp} \, \varphi_v \end{array} } } \varphi_v(\gamma^{-1} \bx) \sum_{h_0 \in (H_{\bx}(F_v) \cap \gamma U_v \gamma^{-1}) \backslash H_{\bx}(F_v)} C_{\gamma U_v}(h_0 h) \  .
\end{align*}
We claim that the outer sum is finite. Since $\mathrm{det} \, T \not= 0$, the group $H(F_v)$ acts transitively on the set $\Omega_{T}=\{ \bv \in V_v^n : Q[\bv] = T \}$ and therefore the map $\gamma \mapsto \gamma^{-1} \bx$ is a homeomorphism between $H_{\bx}(F_v) \backslash H(F_v)$ and $\Omega_{T}$. Since $\varphi_v$ has compact support, its restriction to the closed subset $\Omega_{T}$ has compact support and therefore the support of $\varphi_v(h^{-1} \bx)$ in $H_{\bx}(F_v) \backslash H(F_v)$ is also compact. Finally, since $H_{\bx}(F_v) \backslash H(F_v) / U_v$ is discrete, the set $\{ \gamma \in H_{\bx}(F_v) \backslash H(F_v) / U_v : \gamma^{-1} \bx \in \text{supp} \, \varphi_v \}$ is finite.

We make the observation
\begin{equation}
\int_{H_{\bx}(F_v)} C_{\gamma U_v} (h_0 h) \, dh_0 = \mathrm{vol}(H_{\bx}(F_v) \cap \gamma U_v \gamma^{-1}) \sum_{h_0 \in \left( H_{\bx}(F_v) \cap \gamma U_v \gamma^{-1} \right) \backslash H_{\bx}(F_v)} C_{\gamma U_v}(h_0 h)
\end{equation}
for all $\gamma , h \in H(F_v)$. Now we must \emph{make a choice} $\mathscr{C}$ for a set of representatives $\gamma$ of the double coset space $H_{\bx}(F_v) \backslash H(F_v) / U_v$ and then define a function given by the finite sum
\begin{equation} \label{ProofLocalEq}
\xi_v(h) = \sum_{ {\scriptsize \begin{array}{c} \gamma \in \mathscr{C} \\ \gamma^{-1} \bx \in \text{supp} \, \varphi_v \end{array} } } \frac{\varphi_v(\gamma^{-1} \bx)}{\mathrm{vol}(H_{\bx}(F_v) \cap \gamma U_v \gamma^{-1})} \ C_{\gamma U_v}(h)
\end{equation}
Then $\xi_v$ is a locally constant compactly supported function which matches $\varphi_v$ relative to $\bx$.

We claim that, for almost all $v$, the sum $(\ref{ProofLocalEq})$ has a single term and $\xi_v(h) = C_{K_v}(h)$ where $K_v$ is the maximal compact subgroup $K_v = \mathrm{Aut}(L_v) \subset H(F_v)$ for a fixed global lattice $L \subset V$. For almost all $v$, we are in the following situation:
\begin{enumerate}
\item $\varphi_v = \varphi_{L_v} \otimes \cdots \otimes \varphi_{L_v}$ is the $n$-fold tensor product of the characteristic function of $L_v$
\item $\bx = (x_1 , \dots , x_n) \in L_v^n$
\item $\mathrm{det} \, Q[\bx] \in O_{F_v}^{\times}$
\item $Q$ is $O_{F_v}$-valued on $L_v$
\end{enumerate}
The symmetric matrix $Q[\bx]$ represents the quadratic form $Q$ restricted to $\Lambda_{\bx} = \mathrm{span}_{O_{F_v}}\{ x_1 , \dots , x_n\}$ relative to the basis $x_1 , \dots , x_n$ consisting of the components of $\bx$. Since $\mathrm{det} \, Q[\bx] \in O^{\times}_{F_v}$, the lattice $\Lambda_{\bx}$ is regular therefore $L_v = \Lambda_{\bx} \oplus \Lambda'$ is an orthogonal direct sum of $O_{F_v}$-lattices where
\begin{equation}
\Lambda' = \{ v \in L_v : (v,w)=0\ \text{ for all } w \in \Lambda_{\bx} \} \ . 
\end{equation}
Suppose $\gamma \in H(F_v)$ such that $x_i \in \gamma L_v$ for each $i=1,\dots,n$. Then $\Lambda_{\bx}$ is a regular subspace of $\gamma L_v$ therefore we have the orthogonal direct sum $\gamma L_v = \Lambda_{\bx} \oplus \Lambda''$ of $O_{F_v}$-lattices where
\begin{equation}
\Lambda'' = \{ \gamma v \in \gamma L_v : (\gamma v,w) = 0 \text{ for all } w \in \Lambda_{\bx} \} \ .
\end{equation}
Since $\Lambda_{\bx} \oplus \Lambda'$ and $\Lambda_{\bx} \oplus \Lambda''$ are isometric, we must have that $\Lambda'$ and $\Lambda''$ are isometric by the Witt cancellation property for local rings. In other words, there is some $\delta \in H_{\bx}(F_v)$ (note that $H_{\bx}(F_v)$ is the orthogonal group of $\mathrm{span}_{F_v}\{ x_1 , \dots , x_n\}^{\perp}$) such that $\delta \Lambda'' = \Lambda'$. Finally, $\delta \gamma L_v = L_v$ and so $\delta \gamma \in K_v$. Therefore, in this most unramified case, 
\begin{equation}
\# \{ \gamma \in H_{\bx}(F_v) \backslash H(F_v) / U_v : \gamma^{-1} \bx \in \text{supp} \, \varphi_v \} = 1
\end{equation}
therefore the sum (\ref{ProofLocalEq}) has a single term and $\xi_v(h)=C_{K_v}(h)$. Here we have used the fact that $\mathrm{vol}(H_{\bx}(F_v) \cap K_v)=1$ for almost all $v$.

Suppose $v$ is an infinite place. Then $V_v$ is anisotropic and the group $H(F_v)$ is compact therefore
\begin{equation}
\xi_v(h)= \frac{1}{\mathrm{vol} \, H_{\bx}(F_v) } \, \varphi_v(h^{-1} \bx)
\end{equation}
is a smooth compactly supported function which matches $\varphi_v$ relative to $\bx$.

Finally, suppose $v$ is an infinite place such that $V_v$ is isotropic. This is the only case where we cannot give an exact definition of the local function $\xi_v$. The group $H(F_v)$ is a Lie group, $H_{\bx}(F_v)$ is a Lie subgroup and the natural map $p : H(F_v) \longrightarrow \Omega_{T} : h \mapsto h^{-1} \bx$ is a princial $H_{\bx}(F_v)$-bundle where $\Omega_{T} = \{ \bv \in V_v^n : Q[\bv] = T \}$. Let $\{ U_i , \Phi_i \}_{i \in I}$ be an open cover of $\Omega_T$ with local trivializations $\Phi_i : p^{-1}(U_i) \cong H_{\bx}(F_v) \times U_i$. On each of the open sets $p^{-1}(U_i)$, we can arbitrarily chose a smooth rapidly decreasing function $f_i$ on $H_{\bx}(F_v)$ and define a smooth rapidly decreasing function on $p^{-1}(U_i)$ by $f_i(\gamma) \varphi_v(\bv)$ using the isomorphism $\Phi_i(h)=(\gamma,\bv) \in H_{\bx}(F_v) \times U_i$. We can use a partition of unity to glue the functions $f_i(\gamma) \varphi_v(\bv)$ on each $p^{-1}(U_i)$ together to get a smooth rapidly decreasing function $\xi_v$ on the whole group $H(F_v)$. Since the functions $f_i$ were arbitrary we can scale them so that $\int_{H_{\bx}(F_v)} \xi_v(h_0h) \, dh_0 = \varphi_v(h^{-1}\bx)$.

\end{proof}

\subsection{Example: Local Matching Functions for $\mathrm{PGL}_2(\Qp)$}

In this section, we will compute the local functions $\xi_v$ by the recipe (\ref{MatchingEq2}) in the following particular case. Let $p$ be a prime not equal to 2 and consider the quadratic space $(V,Q)$ over $\Qp$ consisting of traceless 2 by 2 matrices
\begin{equation}
V = \left\{ \begin{pmatrix} a & b \\ c & -a \end{pmatrix} : a,b,c \in \Qp \right\}
\end{equation}
equipped with the quadratic form
\begin{equation}
Q(x) = -\mathrm{det}(x) \ \ (\text{equivalently, } x^2 = Q(x) \cdot \mathrm{id}) .
\end{equation}
Note that this corresponds to the inner product
\begin{equation}
(x,y) = \mathrm{tr}(xy) \ , \ \ x,y \in V \, .
\end{equation}
The special orthogonal group is given by
\begin{equation}
\mathrm{SO}(V) = \mathrm{PGL}_2(\Qp)
\end{equation}
via the natural action $x \mapsto gxg^{-1}$ for $g \in \mathrm{GL}_2(\Qp)$. (Again, we note that even though this paper restricts to the case that $\mathrm{dim}_F \, V$ is even, all our results apply with trivial modifications to the case when $\mathrm{dim}_F \, V$ is odd and $G$ is the metaplectic cover of $\Spn$.)

Let $H=\mathrm{PGL}_2(\Qp)$ and let $K \subset H$ be the maximal compact subgroup which stabilizes the lattice of integral elements
\begin{equation}
L = \left\{ \begin{pmatrix} a & b \\ c & -a \end{pmatrix} : a,b,c \in \Zp \right\} \ .
\end{equation}
In other words, $K$ is equal to the image of $\mathrm{GL}_2(\Zp)$ projected to $H$.

Let $\varphi \in S(V)$ be the characteristic function of the lattice $L \subset V$, let $x \in V$ such that $Q(x) \not= 0$ and let $H_x$ be the stabilizer of $x$ in $H$. The goal of this section is to find a smooth compactly supported function $\xi$ on $H$ such that $\xi$ matches $\varphi$ relative to $x$ as in (\ref{MatchingEq1}). In other words, we want to find $\xi$ such that 
\begin{equation}
\varphi(h^{-1} x) = \int_{H_x} \xi(h' h ) \, dh'
\end{equation}
and we will use the recipe (\ref{MatchingEq2}). Note that since $\varphi$ is an even function, we can extend $\xi$ to a function on the whole orthogonal group $\OV$ by the projection map $\OV \rightarrow \mathrm{SO}(V)$. We begin with a few reduction steps.

Firstly, if $Q(x) \not\in \Zp$ then $\varphi(h^{-1} x) = 0$ for all $h \in H$ since $Q(L) \subset \Zp$. Therefore we may assume $Q(x) \in \Zp$.

Secondly, if $Q(x) = p^{\alpha} \varepsilon$ for $\alpha \in \Z_{\geq 0}$ and $\varepsilon \in \Zp^{\times}$, then $x$ generates a quadratic extension of $\Qp$. Therefore we need only consider three cases (recall, $p > 2$):
\begin{enumerate}
\item[(inert) :] if $\alpha$ is even and $\varepsilon$ is a nonsquare, then we may assume $x = p^{\alpha/2}  \begin{pmatrix} 0 & \varepsilon \\ 1 & 0 \end{pmatrix}$
\item[(ramified) :] if $\alpha$ is odd, then we may assume $x = p^{(\alpha-1)/2} \begin{pmatrix} 0 & p \varepsilon \\ 1 & 0 \end{pmatrix}$
\item[(split) :] if $\alpha$ is even and $\varepsilon$ is a square, then we may assume $x = p^{\alpha/2} \begin{pmatrix} 1 & 0 \\ 0 & -1 \end{pmatrix}$
\end{enumerate}
In particular, if $y \in V$ is any element such that $Q(y)$ is integral and nonzero then $y$ is in the $H$-orbit of one of the elements in the three cases above.

Finally, note that the compact open subgroup $K$ fixes $\varphi$ and therefore it plays the role of $U_v$ appearing in (\ref{MatchingEq2}). Therefore, to compute $\xi$ we need to:
\begin{enumerate}
\item determine a set $\mathscr{C}$ of representatives for $H_x \backslash H / K$ and determine the subset $\{ \gamma \in \mathscr{C} : x \in \gamma L \}$
\item compute $\mathrm{vol}(H_x \cap \gamma K \gamma^{-1})$ for $\gamma \in \{ \gamma \in \mathscr{C} : x \in \gamma L \}$
\end{enumerate}

\begin{lemma} .

\begin{enumerate}

\item The subgroup $H_x$ is equal to the image in $H$ of:
\begin{align*}
\mathrm{(inert)} : & \hspace{2mm} \left\{ \begin{pmatrix} a & b \varepsilon \\ b & a \end{pmatrix} : a,b \in \Qp \, , \, a^2 + \varepsilon b^2 \not= 0 \right\} \\
\mathrm{(ramified)} : & \hspace{2mm} \left\{ \begin{pmatrix} a & b \varepsilon p \\ b & a \end{pmatrix} : a,b \in \Qp \, , \, a^2 + p \varepsilon b^2 \not= 0 \right\} \\
\mathrm{(split)} : & \hspace{2mm} \left\{ \begin{pmatrix} a & 0 \\ 0 & b \end{pmatrix} : a,b \in \Qp \, , \, a b \not= 0 \right\}
\end{align*}

\item A set $\mathscr{C}$ of representatives for the double coset space $H_x \backslash H / K$ is given by:
\begin{align*}
\mathrm{(inert)} : & \hspace{2mm} \mathscr{C} = \left\{ \gamma_d = \begin{pmatrix} p^d & 0 \\ 0 & 1 \end{pmatrix} : d \geq 0 \right\} \\
\mathrm{(ramified)} : & \hspace{2mm} \mathscr{C} = \left\{ \gamma_d = \begin{pmatrix} p^d & 0 \\ 0 & 1 \end{pmatrix} : d \geq 1 \right\} \\
\mathrm{(split)} : & \hspace{2mm} \mathscr{C} = \left\{ \delta_d = \begin{pmatrix} p^d & 1 \\ 0 & 1 \end{pmatrix} : d \geq 0 \right\}
\end{align*}

\item The set $\{ \gamma \in \mathscr{C} : x \in \gamma L \}$ is given by:
\begin{align*}
\mathrm{(inert)} : & \hspace{2mm}  \left\{ \gamma_d = \begin{pmatrix} p^d & 0 \\ 0 & 1 \end{pmatrix} : 0 \leq d \leq \displaystyle \frac{\alpha}{2} \right\} \, , \ \mathrm{and} \ \mathrm{vol}(H_x \cap \gamma_d K \gamma_d^{-1}) = \frac{\mathrm{vol} \, H_x}{p^d+p^{d-1}} \\
\mathrm{(ramified)} : & \hspace{2mm} \left\{ \gamma_d = \begin{pmatrix} p^d & 0 \\ 0 & 1 \end{pmatrix} : 1 \leq d \leq \displaystyle \frac{\alpha+1}{2} \right\} \, , \ \mathrm{and} \ \mathrm{vol}(H_x \cap \gamma_d K \gamma_d^{-1}) = \frac{\mathrm{vol} \, H_x}{2p^{d-1}} \\
\mathrm{(split)} : & \hspace{2mm} \left\{ \delta_d = \begin{pmatrix} p^d & 1 \\ 0 & 1 \end{pmatrix} : 0 \leq d \leq \displaystyle \frac{\alpha}{2} \right\} \, , \ \mathrm{and} \ \mathrm{vol}(H_x \cap \delta_d K \delta_d^{-1}) = \mathrm{vol}(H_x \cap K)
\end{align*}

\end{enumerate}

\end{lemma}

\begin{proof}
We will be brief since all the computations required to prove this lemma are quite elementary. It is a straight forward computation to determine $H_x$ in each case. A set of coset representatives of $H / K$ is given by
$$
\left\{ \begin{pmatrix} p^d & u \\ 0 & 1 \end{pmatrix} : d \geq 0 \ \mathrm{and} \ 0 \leq u < p^d \right\} \cup \left\{ \begin{pmatrix} 1 & 0 \\ u & p^d \end{pmatrix} : d > 0, \ 0 \leq u < p^d \ \mathrm{and} \ p \mid u \right\}
$$
and it is easy to determine $\mathscr{C}$ and $\{ \gamma \in \mathscr{C} : x \in \gamma L \}$ in each case. Finally, we will discuss how to derive the quantity $\mathrm{vol}(H_x \cap \gamma_d K \gamma_d^{-1})$ in the inert case via the action of $H$ on the tree $H / K$. The other cases are proved in a similar way. We may visualize the set $H / K$ as a connected graph with a vertex for each element of $H / K$ which is connected to exactly $p+1$ neighbouring vertices. We measure the distance of a coset $hK$ from the base point $K$ by the number of edges along the shortest path from $K$ to $hK$. It is easy to show that $H_x$ stabilizes $K$ and acts transitively on the $p^d + p^{d-1}$ vertices at a distance $d$ from $K$. The subgroup $\gamma_d K \gamma_d^{-1}$ is the stabilizer of the vertex $\gamma_d K$ therefore $[H_x : H_x \cap \gamma_d K \gamma_d^{-1}] = p^d + p^{d-1}$ and so
$$
\mathrm{vol}(H_x \cap \gamma_d K \gamma_d^{-1}) = \frac{\mathrm{vol} \, H_x}{p^d + p^{d-1}} \ .
$$
\end{proof}

Finally, in this specific case, we have all the ingredients to determine the local matching functions $\xi$ according to the recipe (\ref{MatchingEq2}).

\begin{proposition} \label{MatchingPGL}
If $\varphi$ is the characteristic function of $L$, then, in each of the three cases above, a smooth function $\xi$ on $H$ that matches $\varphi$ is given by:
\begin{align}
\mathrm{(inert)} : & \hspace{2mm} \xi(h) = \frac{1}{\mathrm{vol} \, H_x} \left(C_K(h) + \sum_{d=1}^{\alpha/2} (p^d+p^{d-1}) \, C_{\gamma_d K} (h) \right) \ , \ \ \gamma_d = \begin{pmatrix} p^d & 0 \\ 0 & 1 \end{pmatrix} \\
\mathrm{(ramified)} : & \hspace{2mm} \xi(h) = \frac{1}{\mathrm{vol} \, H_x} \sum_{d=1}^{(\alpha+1)/2} 2 p^{d-1} \, C_{\gamma_d K}(h) \ , \ \ \gamma_d = \begin{pmatrix} p^d & 0 \\ 0 & 1 \end{pmatrix} \\
\mathrm{(split)} : & \hspace{2mm} \xi(h) = \frac{1}{\mathrm{vol}(H_x \cap K)} \left( C_K(h) + \sum_{d=1}^{\alpha/2} C_{\delta_d K} (h) \right) \ , \ \ \delta_d = \begin{pmatrix} p^d & 1 \\ 0 & 1 \end{pmatrix}
\end{align}
where $C_S$ denotes the charactersitic function of a set $S \subset H$.
\end{proposition}

\subsection{The Main Period Identity}

The relation in \autoref{DualPeriods} between Fourier coefficients and orthogonal periods leads to our main identity stated below. Recall, we call a triple $(\varphi , \xi \, ; \bx)$ a \emph{matching datum} (or that \emph{$\xi$ matches $\varphi$ relative to $\bx$}) if $\varphi,\xi,\bx$ satisfy the equation in \autoref{MatchingLemma}. Furthermore, we say that $(\varphi , \xi \, ; \bx)$ is \emph{nondegenerate} if $\mathrm{det} \, Q[\bx] \not= 0$. Therefore, for each nondegenerate matching datum $(\varphi , \xi \, ; \bx)$ with $Q[\bx] = T$, \autoref{DualPeriods} states
$$
W_T(\theta_{\varphi}f^H) = P_{H_{\bx}}(R_{\xi}f^H)
$$
for all cusp forms $f^H \in L^2_{\mathrm{cusp}}(H(F) \backslash H(\A))$.

\begin{theorem} \label{MainTheorem}
Let $\sigma$ be a cuspidal representation of $G$ such that such that $\pi = \theta_{\psi}(\sigma)$ is a cuspidal representation of $H$. Let $(\varphi_1 , \xi_1 \, ; \bx_1)$ and $(\varphi_2 , \xi_2 \, ; \bx_2)$ be a pair of nondegenerate matching data and let  $T_1 = Q[\bx_1]$ and $T_2 = Q[\bx_2]$. Then
\begin{equation} \label{MainEq}
\sum_{F^G \in \mathscr{B}(\sigma)} W_{T_1}(\theta_{\varphi_1} \theta_{\varphi_2^{\vee}} F^G) \, \overline{W_{T_2}(F^G)} \ = \sum_{F^H \in \mathscr{B}(\pi)} P_{H_{\bx_1}}(R_{\xi_1} R_{\xi_2^{\vee}} F^H) \, \overline{P_{H_{\bx_2}}(F^H)}
\end{equation}
where $\mathscr{B}(\sigma)$ and $\mathscr{B}(\pi)$ are orthonormal bases of $\sigma$ and $\pi$ respectively, and $\xi^{\vee}(h)=\overline{\xi(h^{-1})}$.
\end{theorem}

\begin{proof}
First consider the matching datum $(\varphi_2 , \xi_2 \, ; \bx_2)$. By \autoref{DualPeriods} we have
\begin{equation}
W_{T_2}(\theta_{\varphi_2}f^H) = P_{H_{\bx_2}}(R_{\xi_2}f^H)
\end{equation}
for any $f^H \in \pi$. Since $\theta_{\varphi_2}f^H \in \sigma$, we may write $\theta_{\varphi_2}f^H$ in terms of an orthonormal basis of $\sigma$
\begin{equation}
\theta_{\varphi_2}f^H(g) = \sum_{F^G \in \mathscr{B}(\sigma)} \langle \theta_{\varphi_2}f^H , F^G \rangle_G \cdot F^G(g)
\end{equation}
therefore we have
\begin{equation} \label{function1}
W_{T_2}(\theta_{\varphi_2}f^H) = \left\langle f^H(h) \ , \sum_{F^G \in \mathscr{B}(\sigma)} \theta_{\varphi_2^{\vee}}F^G(h) \, \overline{W_{T_2}(F^G)} \right\rangle_H
\end{equation}
where we have used the adjoint property of the theta lift $\langle \theta_{\varphi}f^H , F^G \rangle_G = \langle f^H , \theta_{\varphi^{\vee}}F^G \rangle_H$.

Now write $R_{\xi_2}f^H$ in terms of an orthonormal basis of $\pi$
\begin{equation}
R_{\xi_2}f^H(h) = \sum_{F^H \in \mathscr{B}(\pi)} \langle R_{\xi_2}f^H , F^H \rangle_H \cdot F^H(h)
\end{equation}
therefore
\begin{equation} \label{function2}
P_{H_{\bx_2}}(R_{\xi_2}f^H) = \left\langle f^H(h) \ , \sum_{F^H \in \mathscr{B}(\pi)} R_{\xi_2^{\vee}}F^H(h) \, \overline{P_{H_{\bx_2}}(F^H)} \right\rangle_H
\end{equation}
where we have used the adjoint property $\langle R_{\xi}f^H , F^H \rangle_H = \langle f^H , R_{\xi^{\vee}}F^H \rangle_H$.

Now we have two functions in $\pi$ given by (\ref{function1}) and (\ref{function2}) which represent the same linear functional on $\pi$ and so
\begin{equation} \label{identity1}
\sum_{F^G \in \mathscr{B}(\sigma)} \theta_{\varphi_2^{\vee}}F^G(h) \, \overline{W_{T_2}(F^G)} = \sum_{F^H \in \mathscr{B}(\pi)} R_{\xi_2^{\vee}}F^H(h) \, \overline{P_{H_{\bx_2}}(F^H)} \ .
\end{equation}
Using the other matching datum $(\varphi_1 , \xi_1 \, ; \bx_1)$, apply the linear functional $f^H \mapsto W_{T_1}(\theta_{\varphi_1}f^H)$ to the left hand side of (\ref{identity1}) and, by \autoref{DualPeriods}, the equivalent operation $f^H \mapsto P_{H_{\bx_1}}(R_{\xi_1}f^H)$ to the right hand side of (\ref{identity1}) to produce
\begin{equation}
\sum_{F^G \in \mathscr{B}(\sigma)} W_{T_1}(\theta_{\varphi_1} \theta_{\varphi_2^{\vee}} F^G) \, \overline{W_{T_2}(F^G)} \ = \sum_{F^H \in \mathscr{B}(\pi)} P_{H_{\bx_1}}(R_{\xi_1} R_{\xi_2^{\vee}} F^H) \, \overline{P_{H_{\bx_2}}(F^H)} \ .
\end{equation}
\end{proof}

As a final remark in this section, let us take note of the resemblance between the period identity (\ref{MainEq}) and the identity one would expect from a comparison of relative trace formulas. Informally, if we consider $\theta_{\varphi_1} \theta_{\varphi_2^{\vee}}$ as an operator on the space of cusp forms on $G$ and $R_{\xi_1} R_{\xi_2^{\vee}}$ as an operator on the space of cusp forms on $H$, then the left side of the main period identity above is the relative trace of $\theta_{\varphi_1} \theta_{\varphi_2^{\vee}}$ along $\sigma$ relative to the subgroup $N \times N$ and the characters $\psi_{T_1}$ and $\psi_{T_2}$, and the right side is the relative trace of $R_{\xi_1} R_{\xi_2^{\vee}}$ along $\pi$ relative to the subgroup $T_1 \times T_2$. However, this point of view will not play a role in what follows.

\subsection{An Inner Product Identity} \label{InnerProdId}

Our ultimate goal is to use the main period identity in \autoref{MainTheorem} to relate Fourier coefficients and orthogonal periods of automorphic forms to special values of automorphic $L$-functions. The first step is \autoref{InnerProdThm} below which uses the formal properties of the theta correspondence and the Weil representation to rewrite the main period identity (\ref{MainEq}) in terms of an inner product of theta lifts. In the next section, we will combine this inner product identity with the Rallis inner product formula to obtain our special value formula in \autoref{SpecialValue}.

Let $\sigma$ be a cuspidal automorphic representation of $G$ such that $\pi = \theta_{\psi}(\sigma)$ is a cuspidal representation of $H$. Recall the theta lift from $G$ to $H$ is a $G(\A)$-equivariant map from $\omega^{\vee} \otimes \sigma$ to $\pi$ where the $G(\A)$ acts diagonally on the $\omega^{\vee} \otimes \sigma$ and trivially on $\pi$. Here $\omega^{\vee}$ is the dual Weil representation, or equivalently the Weil representation relative to the character $\psi^{-1}$ (cf.~\autoref{SecWeil}). Therefore it must factor as the composition
\begin{equation}
\xymatrix@R=2mm{
\omega^{\vee} \otimes \sigma \ar[r] & \pi \otimes \sigma^{\vee} \otimes \sigma \ar[r] & \pi \\
\overline{\varphi} \otimes f^G \ar@{|->}[rr] & & \theta_{\varphi^{\vee}} f^G
}
\end{equation}
where the second arrow is a mutiple of the canonical pairing between $\sigma$ and $\sigma^{\vee}$. Therefore, for a fixed $f^G \in \sigma$, there exists a function $\varphi \in S(V(\A)^n)$ such that $\varphi$ \emph{pairs purely} with $f^G$. In other words, there exists $\varphi$ such that $\theta_{\varphi^{\vee}} f^G \not= 0$ and $\theta_{\varphi^{\vee}} F^G \equiv 0$ for all $F^G \in \sigma$ which are orthogonal to $f^G$, ie.~$\langle F^G , f^G \rangle =0$.

If $\varphi$ pairs purely with $f = f^G \in \sigma$, then $f$ is an eigenfunction for the operator $\theta_{\varphi} \theta_{\varphi^{\vee}}$ on $\sigma$. To see this notice, that if $F$ is orthogonal to $f$ then $\theta_{\varphi^{\vee}} F = 0$ and so
$$
\langle \theta_{\varphi} \theta_{\varphi^{\vee}} f , F \rangle_G = \langle \theta_{\varphi^{\vee}} f , \theta_{\varphi^{\vee}} F \rangle_H = 0
$$
by the adjoint property in \autoref{adjoint}. Therefore $\theta_{\varphi}\theta_{\varphi^{\vee}} f$ must be multiple $\lambda$ of $f$. In particular, we have
\begin{equation}
\theta_{\varphi}\theta_{\varphi^{\vee}} f = \lambda f \ \Rightarrow \ \langle \theta_{\varphi}\theta_{\varphi^{\vee}} f , f \rangle = \lambda \langle f , f \rangle  \ \Rightarrow \ \frac{ \langle \theta_{\varphi^{\vee}} f , \theta_{\varphi^{\vee}} f \rangle } { \langle f , f \rangle } = \lambda \ .
\end{equation}

In summary, if we fix some automorphic form $f \in \sigma$ and choose some $\varphi \in S(V(\A)^n)$ which pairs purely with $f$, then \autoref{MainTheorem} reduces to the following inner product identity.

\begin{theorem} \label{InnerProdThm}
Let $\sigma$ be a cuspidal representation of $G$ such that $\pi = \theta_{\psi}(\sigma)$ is a cuspidal representation of $H$. Fix $f \in \sigma$ and let $\varphi \in S(V(\A)^n)$ such that $\varphi$ pairs purely with $f$ as above. Let $(\varphi , \xi_1 \, ; \bx_1)$ and $(\varphi , \xi_2 \, ; \bx_2)$ be a pair of nondegenerate matching data and let  $T_1 = Q[\bx_1]$ and $T_2 = Q[\bx_2]$. Then
\begin{equation} \label{InnerProdId}
\frac{\langle \theta_{\varphi^{\vee}}f , \theta_{\varphi^{\vee}}f \rangle}{\langle f , f \rangle} \, \frac{W_{T_1}(f) \, \overline{W_{T_2}(f)}}{\langle f , f \rangle} \ = \sum_{F^H \in \mathscr{B}(\pi)} P_{H_{\bx_1}}(R_{\xi_1} R_{\xi_2^{\vee}} F^H) \, \overline{P_{H_{\bx_2}}(F^H)}
\end{equation}
where $\mathscr{B}(\pi)$ is an orthonormal basis of $\pi$ and $\xi^{\vee}(h)=\overline{\xi(h^{-1})}$.
\end{theorem}

\subsection{More Period Identities}

The main period identity in \autoref{MainTheorem} is a direct consequence of \autoref{DualPeriods}. In this section, we derive two more period identities from \autoref{DualPeriods}. We will not use these identities in the remainder of this paper however we note that integrals of the form $P_{H_{\bx}}(\theta_{\varphi^{\vee}} f^G)$ appearing below are related (via the Siegel-Weil formula) to the integral representations of automorphic $L$-functions obtained by Piatetski-Shapiro and Rallis in \cite{PSR2}. We also note that an example of such a relation with $L$-functions appears for $n=1$ and $m=3$ in Lemma 45 on p.293 of Waldspurger's work \cite{Wald}.

\begin{theorem}
With the notation and hypotheses of \autoref{MainTheorem}, we have
\begin{align}
\sum_{F^G \in \mathscr{B}(\sigma)} P_{H_{\bx_1}}(\theta_{\varphi_2^{\vee}} F^G) \, \overline{W_{T_2}(F^G)} \ &= \sum_{F^H \in \mathscr{B}(\pi)} P_{H_{\bx_1}}(R_{\xi_2^{\vee}} F^H) \, \overline{P_{H_{\bx_2}}(F^H)} \\
\sum_{F^G \in \mathscr{B}(\sigma)} P_{H_{\bx_1}}(\theta_{\varphi_1^{\vee}} F^G) \ \overline{P_{H_{\bx_2}}(\theta_{\varphi_2^{\vee}} F^G)} \ &= \sum_{F^H \in \mathscr{B}(\pi)} P_{H_{\bx_1}}(\theta_{\varphi_1^{\vee}} \theta_{\varphi_2} F^H) \ \overline{P_{H_{\bx_2}}(F^H)}
\end{align}
\end{theorem}

\begin{proof}
Starting from (\ref{identity1}) in the proof of \autoref{MainTheorem}, the $H_{\bx_1}$-periods of each side of yields the identity
\begin{equation}
\sum_{F^G \in \mathscr{B}(\sigma)} P_{H_{\bx_1}}(\theta_{\varphi_2^{\vee}} F^G) \, \overline{W_{T_2}(F^G)} \ = \sum_{F^H \in \mathscr{B}(\pi)} P_{H_{\bx_1}}(R_{\xi_2^{\vee}} F^H) \, \overline{P_{H_{\bx_2}}(F^H)} \ .
\end{equation}
For the second equation, write $f^G \in \sigma$ in terms of the basis $\mathscr{B}(\sigma)$
\begin{equation}
f^G(g) = \sum_{F^G \in \mathscr{B}(\sigma)} \langle f^G , F^G \rangle_G \, F^G(g)
\end{equation}
therefore
\begin{equation}
P_{T_2}(\theta_{\varphi_2^{\vee}} f^G) = \left\langle f^G , \sum_{F^G \in \mathscr{B}(\sigma)} \overline{P_{T_2}(\theta_{\varphi_2^{\vee}} F^G)} \ F^G\right\rangle_G \, .
\end{equation}
Alternatively, write $\theta_{\varphi_2^{\vee}} f^G$ in terms of the basis $\mathscr{B}(\pi)$
\begin{equation}
\theta_{\varphi_2^{\vee}} f^G (h) = \sum_{F^H \in \mathscr{B}(\pi)} \langle \theta_{\varphi_2^{\vee}} f^G , F^H \rangle_H \, F^H(h)
\end{equation}
and compute $P_{T_2}(\theta_{\varphi_2^{\vee}} f^G)$
\begin{equation}
P_{T_2}(\theta_{\varphi_2^{\vee}} f^G) =  \sum_{F^H \in \mathscr{B}(\pi)} \left\langle \theta_{\varphi_2^{\vee}} f^G , F^H \right\rangle_H \, P_{T_2}(F^H) = \left\langle f^G , \sum_{F^H \in \mathscr{B}(\pi)} \overline{P_{T_2}(F^H)} \, \theta_{\varphi_2} F^H \right\rangle_H
\end{equation}
therefore
\begin{equation}
\sum_{F^G \in \mathscr{B}(\sigma)} \overline{P_{T_2}(\theta_{\varphi_2^{\vee}} F^G)} \, F^G(g) \ = \sum_{F^H \in \mathscr{B}(\pi)} \overline{P_{T_2}(F^H)} \, \theta_{\varphi_2} F^H(g) \, .
\end{equation}
Apply the linear functional $f^G \mapsto P_{T_1}(\theta_{\varphi_1^{\vee}} f^G)$ and we arrive at
\begin{equation}
\sum_{F^G \in \mathscr{B}(\sigma)} P_{T_1}(\theta_{\varphi_1^{\vee}} F^G) \, \overline{P_{T_2}(\theta_{\varphi_2^{\vee}} F^G)} \ = \sum_{F^H \in \mathscr{B}(\pi)} P_{T_1}(\theta_{\varphi_1^{\vee}} \theta_{\varphi_2} F^H) \, \overline{P_{T_2}(F^H)}\, .
\end{equation}
\end{proof}

\section{Special Values of Standard Automorphic $L$-functions} \label{SecValues}
 
Throughout this section, let $G=\Spn$, let $H=\OV$ with $\mathrm{dim}_F \, V = m$ and $\chi = \chi_V$, and let $\sigma$ be a cuspidal automorphic representation of $G$ such that its theta lift $\pi = \theta_{\psi}(\sigma)$ is a nonzero cuspidal automorphic representation of $H$.

Our main result is \autoref{SpecialValue} which is simply the combination of the period identity in \autoref{InnerProdThm} with the Rallis inner product formula (cf.~\autoref{Rallis}). Our formula applies to the case when $G$ and $H$ are either in the convergent range (cf.~(\ref{Convergent})) or the second term range (cf.~(\ref{SecondTerm})). In particular, we have either $m > 2n+1$ or $V$ is anisotropic.

In this section we introduce theta integrals, Eisenstein series, the Siegel-Weil formula (in the convergent range), the doubling method and standard automorphic $L$-functions all for the purpose of stating the Rallis inner product formula in our case. Our treatment is brief since there are already several excellent accounts of the Rallis inner product formula (for example \cite{KR} and \cite{GQT}, and the original work of Rallis \cite{R1}, \cite{R2} and \cite{R3}).

\subsection{Inner Product of Theta Lifts}
Suppose $G$ and $H$ are in the \emph{convergent range}:
\begin{equation} \label{Convergent}
V \text{ is anisotropic or } m - r > 2n+1
\end{equation}
where $r=\mathrm{Witt}(V)$ is the Witt index of $V$ (ie.~the dimension of a maximal isotropic subspace). Note that $r$ marks the position of $H$ in its Witt tower (cf.~(\ref{WittTower})) and since we have restricted our attention to the case where $\sigma$ lifts to $\pi$ which is cuspidal, Rallis's tower property implies that $r \leq 2n$.

In Weil's original work \cite{Weil2} on the Siegel-Weil formula, he proved that theta integrals are absolutely convergent in this range therefore we may rearrange the inner product of theta lifts for $f \in \sigma$ and $\varphi \in S(V(\A)^n)$
\begin{equation} \label{InnerProd}
\langle \theta_{\varphi^{\vee}}f , \theta_{\varphi^{\vee}}f \rangle = \int_{[G \times G]} \left( \int_{[H]} \theta_{\varphi}(g_1,h) \, \overline{\theta_{\varphi}(g_2,h)} \, dh \right) \, \overline{f(g_1)} \, f(g_2) \,  dg_1 \, dg_2 \ .
\end{equation}
The Siegel-Weil formula implies that the inner theta integral is the special value of an Eisenstein series on $\mathrm{Sp}_{2n}$ restricted to the product $G \times G$ as we describe below.

There is the natural embedding of symplectic groups
\begin{equation}
\iota_0 : G \times G \longrightarrow \mathrm{Sp}_{2n} : \left( \begin{pmatrix} a_1 & b_1 \\ c_1 & d_1 \end{pmatrix} , \begin{pmatrix} a_2 & b_2 \\ c_2 & d_2 \end{pmatrix} \right) \mapsto \begin{pmatrix} a_1 & & b_1 & \\ & a_2 & & b_2 \\ c_1 & & d_1 & \\ & c_2 & & d_2 \end{pmatrix}
\end{equation}
and we define the twisted embedding
\begin{equation}
\iota(g_1,g_2) = \iota_0(g_1,\check{g}_2) \ , \ \ \check{g} = \begin{pmatrix} 1_n & \\ & -1_n \end{pmatrix} \, g \, \begin{pmatrix} 1_n & \\ & -1_n \end{pmatrix} \ .
\end{equation}
The twisted embedding translates into the following relation for $\varphi \otimes \overline{\varphi} \in S(V(\A)^{2n})$
\begin{equation}
\omega(\iota(g_1,g_2))(\varphi \otimes \overline{\varphi}) = \omega(g_1)\varphi \cdot \overline{\omega(g_2) \varphi} \ , \ \ g_1, g_2 \in G(\A) \ ,
\end{equation}
where we abuse notation and write $\omega$ for both the Weil representation of $G$ and $\mathrm{Sp}_{2n}$. Therefore the pullback along the twisted embedding $\iota$ of the theta function attached to $\varphi \otimes \overline{\varphi} \in S(V(\A)^{2n})$ has the property
\begin{equation}
\theta_{\varphi \otimes \overline{\varphi}}(\iota(g_1,g_2),h) = \theta_{\varphi}(g_1,h) \, \overline{\theta_{\varphi}(g_2,h)} \ .
\end{equation}

For $s \in \C$, let $I_{2n}(s,\chi)$ be the degenerate principal series representation of $\mathrm{Sp}_{2n}(\A)$ consisting of smooth functions $\Phi(g,s)$ satisfying
\begin{equation}
\Phi(m(a)n(b)g,s) = \chi(\mathrm{det} \, a) |\mathrm{det} \, a|_{\A}^{s+\rho_{2n}} \Phi(g,s) \ , \ \rho_{2n} = \frac{2n+1}{2} 
\end{equation}
for all $m(a) \in M_{2n}(\A)$ and $n(b) \in N_{2n}(\A)$ (where we use the subscript $2n$ to emphasize that here we are considering the Siegel parabolic subgroup of $\mathrm{Sp}_{2n}$). Given $\varphi \in S(V(\A)^n)$, we may define a section $\Phi_{\varphi \otimes \overline{\varphi}}$ of the induced representation $I_{2n}(s,\chi)$ by the formula
\begin{equation} \label{DoubleSection}
\Phi_{\varphi \otimes \overline{\varphi}}(g,s) = \omega(g)(\varphi \otimes \overline{\varphi})(0) \cdot |a(g)|^{s-s_{m,2n}} \ , \ \ s_{m,2n} = \frac{m}{2} - \frac{2n+1}{2}
\end{equation}
where $|a(g)|_{\A} =  |\mathrm{det} \, a |_{\A}$ for $a \in \mathrm{GL}_{2n}(\A)$ in the decomposition $g = m(a) n(b) k \in M_{2n}(\A) N_{2n}(\A) K$ for the maximal compact subgroup $K \subset \mathrm{Sp}_{2n}(\A)$. The corresponding Eisenstein series is the usual sum
\begin{equation} \label{Eisenstein}
E(g,s,\Phi_{\varphi \otimes \overline{\varphi}}) = \sum_{\gamma \in P_{2n}(F) \backslash \Sp_{2n}(F)} \Phi_{\varphi \otimes \overline{\varphi}}(\gamma g , s) \ , \ \ g \in \Sp_{2n}(\A) \ .
\end{equation}
The Eisenstein series $E(g,s,\Phi)$ is absolutely convergent for $\mathrm{Re}(s) > (2n+1)/2$ and has meromorphic continuation to the entire complex plane.

In the convergent range, the Siegel-Weil formula implies that $E(g,s,\Phi_{\varphi \otimes \overline{\varphi}})$ is holomorphic at the special value $s_{m,2n}$ and
\begin{equation}
E(\iota(g_1,g_2),s_{m,2n},\Phi_{\varphi \otimes \overline{\varphi}}) = \frac{\kappa_{m,2n}}{\mathrm{vol} \, [H]} \int_{[H]} \theta_{\varphi}(g_1,h) \, \overline{\theta_{\varphi}(g_2,h)} \, dh
\end{equation}
where
\begin{equation}
s_{m,2n} = \frac{m}{2} - \frac{2n+1}{2} \ \ \text{and} \ \ \kappa_{m,2n} = \left\{ \begin{array}{cl} 1 & \mathrm{if} \ m > 2n+1 \\ 2 & \mathrm{if} \ m \leq 2n+1 \end{array} \right. \ .
\end{equation}
Note that $\kappa_{m,2n} = 2$ in the convergent range only when $V$ is anisotropic.

The Siegel-Weil formula was first proved (for more general dual pairs) by Weil \cite{Weil2} in the case that the critical point $s_{m,2n}$ is in the range of convergence of the Eisenstein series, ie.~$m > 4n+2$. In \cite{KR-SW1} and \cite{KR-SW2}, Kudla and Rallis extended the Siegel-Weil formula for all symplectic-orthogonal dual pairs in the convergent range.

Therefore, in the convergent range, the inner product of theta lifts is given by
\begin{equation} \label{InnerProd2}
\langle \theta_{\varphi^{\vee}}f , \theta_{\varphi^{\vee}} f \rangle = \frac{\mathrm{vol} \, [H]}{\kappa_{m,2n}} \int_{[G \times G]} E(\iota(g_1,g_2),s_{m,2n},\Phi_{\varphi \otimes \overline{\varphi}}) \, \overline{f(g_1)} \, f(g_2) \, dg_1 \, dg_2 \ .
\end{equation}

\subsection{Local Zeta Integrals and $L$-Factors}

Integrals of the form (\ref{InnerProd2}) were studied by Piatetski-Shapiro and Rallis \cite{PSR} in much greater generality using their doubling method to obtain integral representations of standard automorphic $L$-functions attached to representations of classical groups. In the case considered here, their results show that, for $\Phi \in I_{2n}(s,\chi)$ and $f \in \sigma$, the global zeta integral
\begin{equation} \label{Zeta}
Z(s,\Phi,f) = \int_{[G \times G]} E(\iota(g_1,g_2),s,\Phi) \, \overline{f(g_1)} \, f(g_2) \, dg_1 \, dg_2
\end{equation}
reduces to a product of local zeta integrals
\begin{equation}
Z(s,\Phi,f) = \prod_v Z_v(s,\Phi,f) \ .
\end{equation}
The analytic properties of the Eisenstein series imply that $Z(s,\Phi,f)$ also has meromorphic continuation to the entire complex plane. When $f = \otimes_v f_v$ and $\Phi(s) = \otimes_v \Phi_v(s)$ are pure tensors, the local zeta integrals are given by
\begin{equation} \label{LocalZetaInt}
Z_v(s,\Phi,f) = \int_{G(F_v)} \Phi_v(\delta \, \iota(1,g) , s) \, \langle \sigma(g)  f_v , f_v \rangle_v \, dg
\end{equation}
where $\delta \in \Sp_{2n}(\Q)$ is an explicit representative of the `main orbit' (see \cite[p.2-6]{PSR} and \cite[p.76-77]{KR}) and $\langle \, , \rangle = \prod_v \langle \, , \rangle_v$ is a decomposition of the global pairing into a product of local pairings. Furthermore, for all places $v$ where $\sigma_v$ is unramified, the local pairings are normalized so that $\langle f_v^0 , f_v^0 \rangle_v = 1$ for the unramified vector $f_v^0$ used to define the restricted tensor product $\sigma = \otimes_v \sigma_v$.

Let us state the fundamental identity of the doubling method \cite[p.3]{PSR} in this case. Let $S$ be a finite set of places of $F$ containing all archimedean places, all finite places where $\sigma_v$ is ramified and all finite places where $\chi_v$ is ramified. Let $\Phi = \otimes_v \Phi_v \in I_{2n}(s,\chi)$ be a pure tensor such that, for all $v \not\in S$, $\Phi_v$ is $\Sp_{2n}(O_{F_v})$-invariant and normalized so that $\Phi_v(e,s)=1$. Finally, let $f = \otimes_v f_v \in \sigma$ be a pure tensor such that, for all $v \not\in S$, $f_v = f_v^0$ is the unramified vector for which $\langle f_v^0 , f_v^0 \rangle_v = 1$. Then
\begin{equation} \label{PSR}
Z(s,\Phi,f) = \frac{L^S(s+\frac{1}{2},\sigma,\chi)}{b_{2n}^S(s,\chi)} \prod_{v \in S} Z_v(s,\Phi,f)
\end{equation}
where $b_{2n}^S(s,\chi) = \prod_{v \not\in S} b_{2n,v}(s,\chi_v)$ for
\begin{equation}
b_{2n,v}(s,\chi_v) = L_v(s + \rho_{2n},\chi_v) \prod_{k=1}^n \zeta_v(2s+2n+1-2k)
\end{equation}
where $\rho_{2n}=(2n+1)/2$.

\subsection{The Second Term Range}

The \emph{second term range} refers to the case when $G$ and $H$ satisfy:
\begin{equation} \label{SecondTerm}
r > 0 \text{ and } m-r \leq 2n+1 < m \ .
\end{equation}
In the second term range, the theta integral for $\varphi \in S(V(\A)^n)$
\begin{equation}
\int_{[H]} \theta_{\varphi}(g_1,h) \, \overline{\theta_{\varphi}(g_2,h)} \, dh
\end{equation}
no longer converges. However, the inner product of theta lifts $\langle \theta_{\varphi^{\vee}}f , \theta_{\varphi^{\vee}} f \rangle$ is well-defined because of our assumption that both $\sigma$ and $\pi$ are cuspidal.

The Eisenstein series $E(\iota(g_1,g_2),s,\Phi_{\varphi \otimes \overline{\varphi}})$ is still defined as in (\ref{Eisenstein}) and has meromorphic continuation but now may have a pole at the special value $s_{m,2n}$. Recent work of Gan, Qiu and Takeda \cite{GQT} proves the second term identity of the regularized Siegel-Weil formula so that the Rallis inner product formula holds in the second term range as stated in the next section.

\subsection{Rallis Inner Product Formula} \label{Rallis}

The doubling method introduced by Piatetski-Shapiro and Rallis \cite{PSR} and developed by Lapid and Rallis \cite{LR} used the local zeta integrals (\ref{LocalZetaInt}) to define standard local $L$-factors for all admissible representations of simple classical groups and we denote by $L(s,\sigma,\chi)$ the global standard automorphic $L$-function of $\sigma$ twisted by $\chi$. Since the local zeta integrals are equal to local $L$-factors almost everywhere by the fundamental identity (\ref{PSR}), it is common to define the normalized local zeta integrals
\begin{equation} \label{LocalZeta}
Z^*_v(s,\Phi,f) = \frac{ Z_v(s,\Phi,f) }{ L_v(s+\frac{1}{2}, \sigma , \chi) } \ .
\end{equation}
Also, given $\varphi \in S(V(\A)^n)$ and the section $\Phi_{\varphi \otimes \overline{\varphi}}(g,s) \in I_{2n}(s,\chi)$ as in (\ref{DoubleSection}), let us define
\begin{equation}
Z^*_v(s,\varphi,f) = Z^*_v(s,\Phi_{\varphi \otimes \overline{\varphi}},f)
\end{equation}
simply to emphasize the dependence on $\varphi$.

Finally, the Siegel-Weil formula (\ref{InnerProd2}) and the fundamental identity (\ref{PSR}) produce the Rallis inner product formula in the convergent range, and the recent work of Gan, Qiu and Takeda \cite{GQT} establish the Rallis inner product formula in the second term range.

\begin{theorem}[\cite{Weil2}, \cite{PSR}, \cite{KR-SW1}, \cite{KR-SW2}, \cite{GQT}] \label{Rallis}
Let $\sigma$ be a cuspidal representation of $G$ such that such that $\pi = \theta_{\psi}(\sigma)$ is a cuspidal representation of $H$. Suppose $G$ and $H$ are either in the convergent range (cf.~(\ref{Convergent})) or the second term range (cf.~(\ref{SecondTerm})). Then $L(s+1/2, \sigma, \chi)$ is holomorphic at $s=s_{m,2n}$ and given $f \in \sigma$ and $\varphi \in S(V(\A)^n)$ we have
\begin{equation} \label{EqInnerProd}
\langle \theta_{\varphi^{\vee}}f , \theta_{\varphi^{\vee}}f \rangle =  \frac{\mathrm{vol} \ [H]}{\kappa_{m,2n}} \, L(s_{m,2n}+1/2, \sigma, \chi) \, \prod_v Z^*_v(s_{m,2n},\varphi,f)
\end{equation}
where
$$
s_{m,2n} = \frac{m}{2} - \frac{2n+1}{2} \ \ \text{and} \ \ \kappa_{m,2n} = \left\{ \begin{array}{cl} 1 & \mathrm{if} \ m > 2n+1 \\ 2 & \mathrm{if} \ m \leq 2n+1 \end{array} \right. \ .
$$
\end{theorem}

\subsection{A Special Value Formula}

We may now state our final result. As indicated earlier, our special value formula is simply the combination of the period identity in \autoref{InnerProdThm} and the Rallis inner product formula, \autoref{Rallis}.

\begin{theorem} \label{SpecialValue}
Let $\sigma$ be a cuspidal representation of $G$ such that $\pi = \theta_{\psi}(\sigma)$ is a cuspidal representation of $H$. Suppose $G$ and $H$ are either in the convergent range (cf.~(\ref{Convergent})) or the second term range (cf.~(\ref{SecondTerm})). Fix $f \in \sigma$ and let $\varphi \in S(V(\A)^n)$ such that $\varphi$ pairs purely with $f$. Let $(\varphi , \xi_1 \, ; \bx_1)$ and $(\varphi , \xi_2 \, ; \bx_2)$ be a pair of nondegenerate matching data and let  $T_1 = Q[\bx_1]$ and $T_2 = Q[\bx_2]$. Then
\begin{equation}
\frac{\mathrm{vol} \, [H]}{\kappa_{m,2n}} \left\{ L(s_{m,2n}+1/2, \sigma , \chi) \, \prod_v \frac{Z^*_v(s_{m,2n},\varphi,f)}{ \langle f_v , f_v \rangle_v } \right\} \frac{W_{T_1}(f) \, \overline{W_{T_2}(f)}}{\langle f , f \rangle} = \sum_{F^H \in \mathscr{B}(\pi)} P_{H_{\bx_1}}(R_{\xi_1} R_{\xi_2^{\vee}} F^H) \, \overline{P_{H_{\bx_2}}(F^H)}
\end{equation}
where $\mathscr{B}(\pi)$ is an orthonormal basis of $\pi$, $\xi^{\vee}(h)=\overline{\xi(h^{-1})}$ and
$$
s_{m,2n} = \frac{m}{2} - \frac{2n+1}{2} \ \ \text{and} \ \ \kappa_{m,2n} = \left\{ \begin{array}{cl} 1 & \mathrm{if} \ m > 2n+1 \\ 2 & \mathrm{if} \ m \leq 2n+1 \end{array} \right. \ .
$$
\end{theorem}

\bibliographystyle{alpha}
\bibliography{ThetaPeriods}

\end{document}